\documentclass{amsart}
\usepackage[normalem]{ulem}
\usepackage{amsmath}
\usepackage{amsthm}
\usepackage{amsfonts}
\usepackage{amssymb}
\usepackage{indentfirst}
\usepackage{biblatex}
\DefineBibliographyStrings{english}{%
	bibliography = {References}}
\newtheorem{theorem}{Theorem}[section]
\newtheorem{lemma}{Lemma}[section]
\newtheorem{corollary}{Corollary}[section]

\newtheorem{definition}{Definition}[section]

\newtheorem{proposition}{Proposition}[section]
\newtheorem{remark}{Remark}[section]
\newcommand\norm[1]{\left\lVert#1\right\rVert}
\usepackage{mathtools}
\usepackage{thmtools}
\declaretheoremstyle[headfont=\normalfont]{normalhead}
\usepackage{color}
\usepackage{graphicx}
\usepackage{morefloats}
\usepackage{hyperref}
\usepackage{amssymb}
\usepackage{textcomp}
\usepackage{url}
\usepackage{float}
\newtheoremstyle{mydef}
{\topsep}{\topsep}%
{}{}%
{\itshape}{}
{\newline}
{%
	\rule{\textwidth}{1.5pt}\\*%
	\thmname{#1}~\thmnumber{#2}\thmnote{\-\ #3}.\\*[-1.5ex]%
	\rule{\textwidth}{1.5pt}}%
\numberwithin{equation}{section}
\begin{document}
	
\title{2D Thin obstacle problem with data at infinity}
\author{Runcao Lyu}
\address{Department of Mathematical Sciences, University of Science and Technology of China, Hefei, China}
\email{lvruncao@mail.ustc.edu.cn}

\author{Zikai Ye}
\address{Department of Mathematical Sciences, University of Science and Technology of China, Hefei, China}
\email{yzkustc@mail.ustc.edu.cn}

\maketitle

\begin{abstract}
In this paper, we consider the thin obstacle problem in $\mathbb{R}^2$ with data at infinity. We first prove the existence and uniqueness of it. Then we show that its symmetric solutions are actually half-space solutions. Our results are needed when classifying the half-space $(2k-\frac{1}{2})$-homogeneous solutions to the thin obstacle problems in $\mathbb{R}^3$. It is a generalization of one part of Savin-Yu's work \cite{savin2021halfspace} on classifying the half-space $\frac{7}{2}$-homogeneous solutions.
\end{abstract}

\section{Introduction}
Thin obstacle problem, also called Signorini Problem \cite{signorini1959questioni}, is a type of free boundary problem which studies minimizers of the Dirichlet energy
$$
\int_{B_{1}^{+}}|\nabla u|^2
$$
in $\{u\in W^{1,2}(B_{1}^{+})\text{ | }u=g\text{ on }\partial B_1\cap \{x_{n+1}>0\},\text{ }u\geq 0\text{ on } B_1\cap\{x_{n+1}=0\}\},$ where $B_{1}^{+}=B_1\cap \{x_{n+1}>0\}.$ The conditions $u=g\text{ on }\partial B_1\cap \{x_{n+1}>0\},\text{ }u\geq 0\text{ on } B_1\cap\{x_{n+1}=0\}$ are understood in the sense of trace.

By standard variational approach and even reflection, these minimizers satisfy the following Euler-Lagrange equations
\begin{equation}\label{thin}
	\begin{cases}
		\Delta u\leq 0\ \ &in\ B_1\\
		u\geq 0\ \ &in\ B_1 \cap \{x_{n+1}=0\}\\
		\Delta u=0 \ &in\ B_1 \cap (\{u>0\}\cup \{x_{n+1}\neq 0\}).
	\end{cases}
\end{equation}
in a Euclidean ball $B_1$ in $\mathbb{R}^{n+1}$. Since the odd part (with respect to $x_{n+1}$) of $u$ is harmonic and vanishes on $\{x_{n+1}=0\}$, it suffices to consider even solution. It is obvious that a rotation around $x_{n+1}$-axis or positive multiples (we call it normalization) of a solution $u$ to \eqref{thin} is still a solution.

Early results about the regularity of $u$ were given by Richardson \cite{richardson1978variational} and Ural'tseva \cite{ural1987regularity}. Then the optimal regularity of $u$ was proved by Athanaspoulos and Caffarelli \cite{athanasopoulos2006optimal} that $u$ is locally Lipschitz in $B_1$ and locally $C^{1,1/2}$ in $B_1\cap \{x_{n+1}\geq 0\}$. With the help of optimal regularity, several studies were done to study the contact set $\Lambda(u):=\{u=0\}\cap\{x_{n+1}=0\}$ and the free boundary $\Gamma(u):=\partial_{\mathbb{R}^{n}}\Lambda(u)$ as follows.

By making use of Almgren's frequency and monotonicity formula \cite{almgren1979dirichlet}, Athanasopoulos-Caffarelli-Salsa \cite{athanasopoulos2008structure} showed that for any $q\in \Lambda(u)$, there is a constant $\lambda_q$ such that 
$$
\norm{u}_{L^2(\partial B_r(q))}\sim r^{\frac{n}{2}+\lambda_q}
$$
as $r\to 0$ and up to subsequence,
$$
u_{q,r}:=r^{\frac{n}{2}}\frac{u(r\cdot+q)}{\norm{u}_{L^2(\partial B_r(q))}}\to u_0.
$$
as $r\to 0$. The constant $\lambda_q$ is called the frequency of $u$ at $q$ and $u_0$ is a $\lambda_q$-homogeneous solution to \eqref{thin}, called a blow-up profile of $u$ at $q$. A constant $\lambda\in\mathbb{R}$ is called an admissible frequency if it is a frequency of some solution $u$ to \eqref{thin} at some point $q\in\Lambda(u)$. There are many results on the classification of admissible frequencies
$$
\Phi:=\{\lambda\in\mathbb{R}\text{ | there is a non-trivial } \lambda \text{-homogeneous solution to \eqref{thin}}\}
$$
and $\lambda$-homogeneous solutions
$$
\mathcal{P}_\lambda:=\{u\text{ | }u\text{ is a solution to \eqref{thin} with }x\cdot \nabla u=\lambda u\}
$$
for each admissible frequency $\lambda\in\Phi$.

The classification in the case of $\mathbb{R}^2$ is completed. See Petrosyan-Shahgholian-Ural'tseva \cite{petrosyan2012regularity}. The set of all admissible frequencies is 
$$
\Phi=\mathbb{N}\cup\{2k-\frac{1}{2}\text{ | }k\in \mathbb{N}\}.
$$
The homogeneous solutions with integer frequencies, up to normalization, are actually even reflections of polynomials. As for the $(2k-\frac{1}{2})$-homogeneous solutions, up to normalization, one has
$$
\mathcal{P}_{2k-\frac{1}{2}}=\{au_{2k-\frac{1}{2}}:=ar^{2k-\frac{1}{2}}\cos((2k-\frac{1}{2})\theta)\text{ | }a\geq 0\},
$$
where $(r,\theta)$ is the polar coordinate of $\mathbb{R}^2$. They vanish on the half line $\{x_1\geq0\text{, }x_2=0\}$ and satisfy
$$
\text{spt}(\Delta u)=\Lambda(u)=\{x_1\leq0\text{, }x_2=0\}
$$

In higher dimensions, the classifications of admissible frequencies and homogeneous solutions remain open. Athanasopoulos-Caffarelli-Salsa \cite{athanasopoulos2008structure} showed that the set of admissible frequencies satisfies
$$
\mathbb{N}\cup\{2k-\frac{1}{2}\text{ | }k\in \mathbb{N}\}\subset \Phi \subset \{1,\frac{3}{2}\}\cup [2,+\infty),
$$
while the left-hand side is obtained by extending the homogeneous solutions from $\mathbb{R}^2$. Moreover, Columbo-Spolaor-Velichkov \cite{colombo2020direct} and Savin-Yu \cite{savin2021fine} showed that there is a frequency gap around each integer.

When it comes to the classification of homogeneous solutions, the case of integral frequencies are studied by Figalli-Ros-Oton-Serra \cite{figalli2020generic} and Garofalo-Petrosyan \cite{garofalo2009some}. It is shown by Athanasopoulos-Caffarelli-Salsa that 
$$
\mathcal{P}_{\frac{3}{2}}=\{\text{Normalizations of } u_{\frac{3}{2}}\}.
$$

For the homogeneous solutions with frequency $2k-\frac{1}{2}$, $k=2,3,...$, it is still tough now to classify them. However, for half-space solution $u$ of frequency $\frac{7}{2}$ in $\mathbb{R}^3$, i.e. up to normalization,
$$
\text{either spt}(\Delta u)\subset \{x_2\leq 0\text{, }x_3=0\} \text{, or spt}(\Delta u)\supset \{x_2\leq 0\text{, }x_3=0\},
$$
all the half-space solutions $u$ of frequency $\frac{7}{2}$ in $\mathbb{R}^3$ we currently know consist of (up to normalization)
$$
\mathcal{F}_1=\{u_{\frac{7}{2}}+a_1x_1u_{\frac{5}{2}}+a_2(x_{1}^2-\frac{1}{5}r^2)u_{\frac{3}{2}}\text{ | }0\leq a_2\leq 5\text{, }a_{1}^2\leq\Gamma(a_2)\},
$$
where
$$
\Gamma(a_2):=\text{min}\{4(a_2(1-\frac{1}{5}a_2)\text{, }\frac{24}{25}a_2(\frac{7}{2}-\frac{3}{10}a_3))\}\text{, }
$$
$$
(r,\theta)\text{ is the polar coordinate to }(x_2,x_3)\text{ - plane}.
$$

Savin-Yu \cite{savin2021halfspace} proved that all $\frac{7}{2}$-homogeneous solutions that are close to $\mathcal{F}_1$ actually lies in $\mathcal{F}_1$ (up to normalization). Then they characterized the rate of convergence of blow-up and proved the regularity of the contact set $\Lambda_{\frac{7}{2}}$ which consists of all points in $\Lambda$ with frequency $\frac{7}{2}$. The main idea of their proof is to use the homogeneity of $u$ and reduce thin obstacle problem in $B_1$ to $\mathbb{S}^2$. Since all the difficulties in the thin obstacle problem occur near the free boundary, we  need to consider thin obstacle problem around $\mathbb{S}^2\cap \{r=0\}$, that is,  near east and west poles. Direct calculation implies that the corresponding solutions on these opposite spherical caps should satisfy a kind of (almost) symmetry.  At the infinitesimal level, it reduces to the problem in 2D with data at infinity and the corresponding two solutions should be (almost) symmetric.

Similar result for general half-space $(2k-\frac{1}{2})$-homogeneous solutions remains unknown. But motivated by Savin-Yu's method \cite{savin2021halfspace}, it is clear from above that if we want to study half-space $(2k-\frac{1}{2})$-homogeneous solutions for $k\geq 2$, we need to consider the solution to the 2D thin obstacle problem with data at infinity that is close to a linear combination of $u_{l-\frac{1}{2}}$, $l=1,2,...,2k,$ which is the main purpose of this paper.

As a building block, we first prove the existence and uniqueness of 2D thin obstacle problem. The uniqueness is shown by maximum principle. To prove the existence, we construct a barrier function and meanwhile, we estimate finer expansion of the solution $u$. With the building block, we find the Fourier coefficients vanish on sufficiently large circles. The corresponding statement on small spherical caps imply the boundedness of the solution. For the case of frequency $\frac{7}{2}$, see Appendix A in Savin-Yu \cite{savin2021halfspace}.

Then we characterize a pair of symmetric solutions to the 2D thin obstacle problem. We first construct two auxiliary odd symmetric polynomials to indicate spt($\Delta u$), the support of $\Delta u$ as a measure. Using these polynomials, we argue by contradiction to prove that spt($\Delta u$) has only one connected component in $\mathbb{R}$, that is to say, $u$ is a half-space solution. In Savin-Yu \cite{savin2021halfspace}, they also argued by contradiction but they carefully counted the number of each zero of these polynomials to get a contradiction with the degree of these polynomials. However, not all zeros can be found in the general case. As a consequence, their method does not work in general. Instead, we focus on the odevity of the multiplicities of certain zeros to obtain a contradiction, which does not have to take all zeros into account. This is the main improvement of our work. After that, we finish our characterization and extend it to the case of almost symmetric solutions, which is truly needed for the spherical case.

This paper is organized as follows:

In Section 2, we give detailed statement of 2D thin obstacle problem with data at infinity and main results. In Section 3, some preliminaries for the paper are introduced. In Section 4, we prove the existence and uniqueness of the problem. Moreover, we show that the Fourier coefficients of the solution on large circles vanish. In Section 5, we characterize a pair of symmetric solutions and extend it to almost symmetric case.

\section{Main Results}
To classify two blow-up solutions around opposite spherical caps, Savin-Yu \cite{savin2021halfspace} constructed two polynomials and counted the number of zeros of these polynomials. However, this method does not work if we study half-space $(2k-\frac{1}{2})$-homogeneous solutions for $k\geq 4$. 

In this paper, we extend the results of \cite{savin2021halfspace} on 2D thin obstacle problem with data at infinity to the general case. To be more specific, instead of counting the number of zeros of the polynomials, we check the odevity of the multiplicities of certain zeros of these polynomials to get a contradiction, which works for general cases.

Recall the definition from \eqref{basedef}. For $p=u_{2k-\frac{1}{2}}+\sum\limits_{l=1}^{2k-1}a_lu_{2k-\frac{1}{2}-l}$, we study the solutions to the thin obstacle problem in $\mathbb{R}^2$ with data $p$ at infinity:
\begin{equation}\label{2D}
	\begin{cases}
		&u \text{ solves \eqref{thin} in } \mathbb{R}^2,\\ 
		&\sup\limits_{\mathbb{R}^2}|u-p|<+\infty.
	\end{cases}
\end{equation} 

\begin{remark}\label{basedef}
Similar to the definition of $u_{2k-\frac{1}{2}}$ in $\mathbb{R}^{n+1}$, we can define $u_{k-\frac{1}{2}}(r,\theta):=r^{k-\frac{1}{2}}\cos((k-\frac{1}{2})\theta)$ for all $k\in \mathbb{Z}$, where $(r,\theta)$ is the polar coordinate in $(x_n,x_{n+1})$ We say $u$ solves \eqref{thin} in $\mathbb{R}^2$ if $u$ solves \eqref{thin} in $B_R$, for any $R>0$.
\end{remark}

With similar approach as in Savin-Yu \cite{savin2021halfspace}, we prove the existence and uniqueness of \eqref{2D} as a building block. We also show that the Fourier coefficients of the solution vanishes along big circles, which is needed to bound the solution of thin obstacle problem on spherical caps. 

\begin{theorem}\label{start}
For $|a_l|\leq1,\ l=1,...,2k-1$, there is a unique solution $u$ to \eqref{2D}. Moreover, for any $N>0$, there are $b_j\in \mathbb{R}$. $1\leq j\leq N$ such that $|b_j|\leq M$, 
$$
\biggl|u-\biggl(p+\sum_{j=1}^{N}b_ju_{\frac{1}{2}-j}\biggr)\biggr|\leq M|x|^{-N}u_{-\frac{1}{2}}\ \text{for all }\ x\in\mathbb{R}^2
$$
and
$$
\int_{\partial B_R}\biggl[u-\biggl(p+\sum_{j=1}^{N}b_ju_{\frac{1}{2}-j}\biggr)\biggr]\cdot\cos((l-\frac{1}{2}))=0
$$ 
for all $1\leq l\leq N\text{ and }R\geq M$, where $M>0$ is a universal constant.
\end{theorem}
\begin{remark}
We will denote $b_j=b_j[a_1,...,a_{2k-1}]$.	
\end{remark}

Before we state our second result, for simplicity, we shall introduce some definitions. We first define a transform operator $U_{\tau}$ as
\begin{equation}\label{transform}
	U_{\tau}(x_1,x_2)=(x_1+\tau,x_2),\ U_{\tau}(f)(x)=f(U_{-\tau}x).
\end{equation}

Then we give the definitions of conjugacy and symmetry
\begin{definition}

(i) For each $p$ that can be written as the form $p=u_{2k-\frac{1}{2}}+\sum\limits_{l=1}^{2k-1}a_lu_{2k-\frac{1}{2}-l}$. We say $q$ is conjugate to $p$ if
$$
q=u_{2k-\frac{1}{2}}+\sum\limits_{l=1}^{2k-1}(-1)^la_lu_{2k-\frac{1}{2}-l.
}$$
(ii) We say $u$ and $v$ are symmetric if there exists a constant $\tau$ such that $U_{\tau}(u)$ is conjugate to $U_{-\tau}(v)$. 
\end{definition}
\begin{remark}
From the definition, we see that in particular, symmetric solutions are half-space solutions.
\end{remark}

The following theorem says that under anti-symmetric conditions, two solutions to 2D thin obstacle problem with conjugate data at infinity are half-space solutions and symmetric. As stated in the introduction, if we restrict the solutions of \eqref{thin} to two opposite spherical caps in $\mathbb{S}^2$, we find that the conjugacy of data at infinity and symmetry of solutions are required. For instance, for the case of $k=2$, we can consider
$$
v_{-\frac{1}{2}}=(x_{1}^{4}-6x_{1}^2r^2-r^4)u_{-\frac{1}{2}}\text{, and }v_{-\frac{3}{2}}:=(x_{1}^{5}+10x_{1}^{3}r^2-15x_1r^4)u_{\frac{3}{2}},
$$
where $(r,\theta)$ is understood as the polar coordinate of $(x_2,x_3)$. Both of them are $\frac{7}{2}$-homogeneous functions. Roughly speaking, if we restrict them to opposite spherical caps near $(\pm1,0,0)$ infinitesimally, we can view them as the counterparts of $u_{-\frac{1}{2}}$ and $u_{-\frac{3}{2}}$ in $\mathbb{R}^2$ and obtain the anti-symmetry conditions of $b_i$ by their odd symmetry with respect to $\{x_1=0\}$. Similarly, we can obtain the anti-symmetry condition for $a_l$.

\begin{theorem}\label{SymSolIn2D}
Given the function $p=u_{2k-\frac{1}{2}}+\sum\limits_{l=1}^{2k-1}a_lu_{2k-\frac{1}{2}-l}$ and let $q$ be the conjugate function of $p$ with $|a_l|\leq 1.$
Suppose that $u$ and $v$ are solutions to \eqref{2D} with the datum $p$ and $q$ at infinity respectively. Assume $b_i[a_1,a_2,\dots,a_{2k-2},a_{2k-1}]=(-1)^{i+1}b_i[-a_1,a_2,\dots,a_{2k-2},-a_{2k-1}]$, for all $1\leq i\leq 2k-2$, then $u$ and $v$ are symmetric.

Moreover, we can find universally bounded constants $\alpha_{l},\ 1\leq l\leq 2k-2$ and $\tau$ such that 
$$
u=U_\tau\biggl(u_{2k-\frac{1}{2}}+\sum\limits_{l=1}^{2k-2}\alpha_lu_{2k-\frac{1}{2}-l}\biggr),\text{ }v=U_{-\tau}\biggl(u_{2k-\frac{1}{2}}+\sum\limits_{l=1}^{2k-2}(-1)^l\alpha_lu_{2k-\frac{1}{2}-l}\biggr),
$$
\end{theorem}

\section{Preliminaries}
In this section, we first introduce some notations, which is similar to \cite{savin2021halfspace}. Unless there is a special explanation, we will consider the case of $\mathbb{R}^2$ in this paper. Denote the polar coordinates of  $(x_{1},x_{2})$ by $(r,\theta)$. Then the slit is defined as
\begin{equation}\label{slit}
	\mathcal{S}:=\{\theta=\pi\}=\{x_1\leq 0, x_{2}=0\}.
\end{equation}

For a subset of $\mathbb{R}^{2}$, we can decompose it as $E=\widehat{E} \cup \widetilde{E}$, where $\widehat{E}=E\setminus\mathcal{S},$ and $\widetilde{E}=E\cap\mathcal{S}$. Given a domain $\Omega\subset\mathbb{R}^{n+1}$, a \textit{harmonic function in the slit domain $\widehat{\Omega}$} 
is a continuous  function that is even with respect to $\{x_{n+1}=0\}$ and satisfies 
\begin{equation}\label{HarmonicInSlit}
	\begin{cases}
		\Delta v=0 &\text{ in $\widehat{\Omega}$,}\\
		v=0 &\text{ in $\widetilde{\Omega}$.}
	\end{cases}
\end{equation} 

It is easy to check $u_{k-\frac{1}{2}}$ is harmonic in $\widehat{\mathbb{R}^2}$. We denote the derivatives of $u_{2k-\frac{1}{2}}$ by the following (for simplicity, we also denote $w_{2k-\frac{1}{2}}:=u_{2k-\frac{1}{2}}$):
$$
w_{2k-\frac{3}{2}}:=\frac{\partial}{\partial x_1}u_{2k-\frac{1}{2}}, w_{l-\frac{1}{2}}:=\frac{\partial}{\partial x_1}w_{l+\frac{1}{2}}\text{, }l\in \mathbb{Z},\ l\leq 2k-1.
$$

Given a non-negative integer $m$, we define the following class of $(m+\frac{1}{2})$-homogeneous functions
\begin{equation}\label{class}
	\mathcal{H}_{m+\frac 12}:=\{v: v \text{  is a harmonic function in }\widehat{\mathbb{R}^{2}}, \quad x\cdot\nabla v=(m+\frac 12)v\}.
\end{equation} 

The case in dimension $2$ is rather simple. 
Its proof follows easily by writing $v\in\mathcal{H}_{m+\frac{1}{2}}$ in the polar coordinate. Each element of $\mathcal{H}_{m+\frac{1}{2}}$ is a multiple of $u_{m+\frac{1}{2}}$:
$$
\mathcal{H}_{m+\frac{1}{2}}=\{au_{m+\frac{1}{2}}\text{ | }a\in\mathbb{R}\}.
$$
Then we can approximate the solution of \eqref{HarmonicInSlit} in $B_1$ by functions in $\mathcal{H}_{m+\frac{1}{2}}$.  

\begin{theorem}[Theorem 4.5 from \cite{Silva2014CinftyRO}]\label{TaylorExpansion}
	Let $v$ be a solution to \eqref{HarmonicInSlit} with $\Omega=B_1\subset \mathbb{R}^n$ and $\|v\|_{L^\infty(B_1)}\le 1$. Given $N\geq 0$, we can find  $v_{m+\frac{1}{2}}\in\mathcal{H}_{m+\frac{1}{2}}$ for $m=0,1,\dots, N$, such that $$\|v_{m+\frac {1}{2}}\|_{L^\infty(B_1)}\leq C$$ and
	$$
	\biggl|v-\sum_{m=0}^{N}v_{m+\frac{1}{2}}\biggr|(x)\leq C|x|^{N+1}u_{\frac {1}{2}} \text{ for $x\in B_{\frac{1}{2}}$,}
	$$where $C$ depends only on $m$.
\end{theorem} 
\begin{remark}
	In the 2D case, this theorem can be directly derived by using a square-root transformation and the Taylor expansion near $0$.
\end{remark}

In the following two sections, we will generalize the results of Appendix B in \cite{savin2021halfspace} from data $p=u_{\frac{7}{2}}+a_1 u_{\frac{5}{2}}+a_2 u_{\frac{3}{2}}+a_3 u_{\frac{1}{2}}$ to the general case $p=u_{2k-\frac{1}{2}}+\sum\limits_{l=1}^{2k-1}a_lu_{2k-\frac{1}{2}-l}$. 

\section{Existence and Uniqueness}
In this section, we prove Theorem \ref{start}. First, we prove the existence and uniqueness of the solution to the thin obstacle problem in $\mathbb{R}^2$ with data $p$ at infinity \eqref{2D}. Then as a corollary, we deduce that Fourier coefficients of the solution vanish along big circles. Recall that $p=u_{2k-\frac{1}{2}}+\sum\limits_{l=1}^{2k-1}a_lu_{2k-\frac{1}{2}-l}$.

\begin{proposition}[Uniqueness]\label{uniqueness}
	Let $|a_j|\leq 1$, if there are two solutions $u_1$, $u_2$ to the system \eqref{2D}, then $u_1=u_2$.
\end{proposition}
\begin{proof}
	Suppose that $u_1$ and $u_2$ are two solutions to \eqref{thin} in $\mathbb{R}^2$ with $\sup_{\mathbb{R}^2}|u_j-p|<+\infty.$ We first claim that there exists a constant $M>0$ such that 
	$$
	\Delta u_j=0 \text{ in }\widehat{\{r>M\}}, \text{ and } u_j=0 \text{ in }\widetilde{\{r>M\}}.
	$$
	
	To derive the first part of the claim, we point out that
	$$\sup_{\mathbb{R}^2}|u_j-p|\leq C$$
	
	However, for $M$ big enough, one has $p>C$ on $\{x_1> M, x_2=0\}$, which implies that $u_j>0$ on $\{x_1> M, x_2=0\}$ and the first part of the claim follows. 
	
	For the second part, we first briefly summarize the idea of the proof. We will construct a function $\varphi$ such that it is also a solution to the thin obstacle problem and $\varphi=0$ at $(x_0,0)$ with some fixed $x_0\leq -M$. Moreover, we show that $\varphi\geq p+C$ along the boundary of a neighborhood $\Omega$ of $(x_0,0)$. Then, we have $u_j\leq p+C\leq \varphi$ on $\partial\Omega$, and the maximum principle implies that $u_j\leq 0$ at $(x_0,0)$, which together with the definition of $u_j$ yields $u_j=0$ at $(x_0,0)$.
	
	Now, it suffices to show that for any $(x_0,0)$ where $x_0\leq -M$ , we can construct $\varphi$. First, we set the neighborhood of $(x_0,0)$ by
	$$
	\Omega:= \{(x_1,x_2):|x_1-x_0|\leq d,|x_2|\leq d \} 
	$$
	with $d$ to be determined latter. We now define 
	$$
	\varphi(x_1,x_2):=(|x_1-x_0|^2-|x_2|^2)/d^3.
	$$
	
	It follows directly that $\varphi$ is a global solution to the thin obstacle problem. and now we would show
	$$
	\varphi\geq p+C\hspace{5pt} on \hspace{5pt}\partial\Omega.
	$$
	By even symmetry, it suffices to show it holds on $\partial\Omega\cap\{x_2\geq 0\}$. For that, we point out that
	$$
	\partial_{x_2}p\leq -c|x_0|^{2k-\frac{1}{2}}
	$$
	for $M$ big enough, where $c$ is a positive universal constant. Inside $\Omega^+=\overline{\Omega}\cap\{x_2\geq0\}$, one has
	\begin{align*}
		\varphi-p&\geq (|x_1-x_0|^2-|x_2|^2)/d^3+c|x_0|^{2k-\frac{1}{2}}x_2\\ &\geq(|x_1-x_0|^2)/d^3+
		\frac{c}{2}|x_0|^{2k-\frac{1}{2}}x_2
	\end{align*}
	
	When $|x_1-x_0|=d\text{ , }0\leq x_2\leq d$, we have
	$$
	\varphi-p\geq \frac{1}{d}+\frac{c}{2}|x_0|^{2k-\frac{1}{2}}x_2\geq \frac{1}{d}\geq C,
	$$
	if we choose $d$ small enough. When $|x_1-x_0|\leq d\text{ , }|x_2|=d$, we have
	$$
	\varphi-p\geq \frac{c}{2}|x_0|^{2k-\frac{1}{2}}x_2\geq \frac{c}{2}dM^{2k-\frac{1}{2}}\geq C,
	$$
	if $M$ is sufficiently large. By even symmetry, we have $\varphi-p\geq C\text{ on }\partial \Omega$ and we finish the proof of the claim. Then, we define
	$$w(x)=(u_1-u_2)\biggl(\frac{M^2x}{|x|^2}\biggr)$$
	to be the Kelvin transform of $(u_1-u_2)$ with respect to $\partial B_M$. Then $w$ is a bounded harmonic function in $\widehat{B_M}$. By Theorem \ref{TaylorExpansion} with $N=0$, we have $|w|\leq Cu_{\frac{1}{2}}$, which implies
	$$
	|u_1-u_2|\leq Cu_{-\frac{1}{2}}\hspace{4pt}in\hspace{4pt}\mathbb{R}^2
	$$
	by inverting Kelvin transform. By maximum principle, we conclude that $u_1=u_2$. 
\end{proof}

Next we construct a barrier function which will be used to prove the existence.

\begin{lemma}[Barrier function]\label{barrier}
Let $|a_j|\leq 1$ for any $1\leq j\leq2k-1$. Then there is a solution $Q$ to \eqref{thin} such that
$$
Q\geq p\text{ on }\{r\geq M\}
$$
for a universally large constant $M$.
\end{lemma}
\begin{proof}
Rewrite $p$ in the basis $\{u_{2k-\frac{1}{2}},\dots,w_{\frac{1}{2}}\}$ as $p=u_{2k-\frac{1}{2}}+\sum\limits_{l=1}^{2k-1}\tilde{a}_lu_{2k-\frac{1}{2}-l}$. For $\tau>0$ to be chosen, we denote $(\alpha_1,\alpha_2,...,\alpha_{2k-2})$  to be the solution to the system
$$
\sum_{j=0}^{i}\frac{1}{j!}\alpha_{i-j}\tau^j=\tilde{a}_i,\text{ }1\leq i\leq 2k-2
$$
Here we take $\alpha_0=1$.
For example, when $k=3$, it's
\begin{align*}
	&\alpha_1+\tau=\tilde{a}_1,\ \alpha_2+\alpha_1\tau+\frac{1}{2}\tau^2=\tilde{a}_2,\ \alpha_3+\alpha_2\tau+\frac{1}{2}\alpha_1\tau^2+\frac{1}{6}\tau^3=\tilde{a}_3,\\ &\alpha_4+\alpha_3\tau+\frac{1}{2}\alpha_2\tau^2+\frac{1}{6}\alpha_1\tau^3+\frac{1}{24}\tau^4=\tilde{a}_4.
\end{align*}
We define
$$
q=u_{2k-\frac{1}{2}}+\sum\limits_{l=1}^{2k-2}\alpha_lu_{2k-\frac{1}{2}-l}.$$ 
Then Taylor's Theorem gives
$$
U_{-\tau}(q)-p\geq \biggl(\frac{1}{(2k-1)!}\tau^{2k-1}+\sum_{l=1}^{2k-2}\frac{\alpha_l}{(2k-1-l)!}\tau^{2k-1-l}-\tilde{a}_{2k-1}\biggr)w_{-\frac{1}{2}}-C\tau^{2k}w_{-\frac{1}{2}}.
$$
Choosing $\tau$ large universally and invoking the definition of $U_\tau$ from \eqref{transform}, then 
$$
U_{-\tau}(q)-p\geq 0 \text{ on } \{r\geq M\}
$$ for a universally large $M$. By choosing $\tau$ larger, if necessary, it is straightforward to verify that $(\alpha_1,\dots,\alpha_{2k-2})$ satisfies the  condition which allows $q$ to solve the thin obstacle problem in $\mathbb{R}^2$, and consequently, $Q:=U_{-\tau}q$ solves the thin obstacle problem in $\mathbb{R}^2.$
\end{proof}

Then we can prove the existence of the 2D thin obstacle problem with data $p$ at infinity.

\begin{proposition}\label{existence}
Let $|a_j|\leq 1$, for any $1\leq j\leq2k-1$. Then there is a unique solution $u$ to the 2D thin obstacle problem with data $p$ at infinity. And there is a universal constant  $M>0$ such that
$$
\sup_{\mathbb{R}^2}u\leq M,\ \Delta u=0 \text{ in } \widehat{\{r>M\}},\text{ and } u=0\text{ on } \widetilde{\{r>M\}}.
$$
Moreover, for any $N$, we can find $b_1,...,b_N$ satisfying $|b_j|\leq M$ such that
$$
\biggl|u-\biggl(p+\sum_{j=1}^{N}b_ju_{\frac{1}{2}-j}\biggr)\biggr|\leq M|x|^{-N}u_{-\frac{1}{2}}\ \text{for all }\ x\in\mathbb{R}^2.
$$
\end{proposition}

\begin{proof}
For large $n\in\mathbb{N}$, let $u_n$ be the solution to the thin obstacle problem \eqref{thin} in $B_n$ with $u_n=p$ along $\partial B_n$. By the maximum principle, we have $u_n\geq p\text{ in }B_n$. Also, since $Q$ and $u_n$ are solutions to the thin obstacle problem in $B_n$, we have
$$
\Delta(Q-u_n)^{-}\geq 0\text{ in }B_n\text{ and } (Q-u_n)^{-}=0\text{ on }\partial B_n,
$$
which implies that $(Q-u_n)^-\leq 0\text{ in }B_n$ i.e. $u_n\leq Q\text{ in }B_n$. Thus we have
\begin{equation}\label{B2} 
p\leq u_n\leq Q \text{ in }B_n
\end{equation} 
if $n$ is large. This gives us a locally uniformly bounded family $\{u_n\}$. Then up to subsequence, it converges to some $u_\infty$ locally uniformly on $\mathbb{R}^2$, which is a solution to the thin obstacle problem in $\mathbb{R}^2$.
	
From \eqref{B2}, we have  $u_n=0 \text{ in }B_n\cap\{x_1\leq -M, x_2=0\}$ and $u_n\geq 1 \text{ in }B_n\cap\{x_1\geq M, x_2=0\}$ for a universal $M>0$. Thus we have 
$$
\Delta u_\infty=0 \text{ in }\widehat{\{r>M\}}; \text{ and }u_\infty=0 \text{ in  }\widetilde{\{r>M\}}.
$$
On $\{r=M\}$, we have $0\leq u_\infty-p\leq Q-p\leq C$. Thus the maximum principle, applied to the domain $\{r>M\}$, gives
$$
|u_\infty-p|\leq C
$$ for a universal constant $C$.
In particular, $u_\infty$ is the unique solution to \eqref{2D}, according to Proposition \ref{uniqueness}. For simplicity, we then denote $u_{\infty}$ by $u$.

Next we deduce a refined expansion of the solution $u$. Let $w(x):=(u-p)(\frac{M^2x}{|x|^2})$ be the Kelvin transform of $u-p$ with respect to $\partial B_M$. Results from previous steps imply that $w$ is a harmonic function in the slit domain $\widehat{B_M}$. Applying Theorem \ref{TaylorExpansion} with $N-1$, we get universally  bounded $b_j,\ j=1,...,N,$ such that 
$$
\biggl|w-\sum_{j=1}^{N}b_j u_{j-\frac{1}{2}}\biggr|\leq C|x|^{N}u_{\frac{1}{2}} \text{ in }B_M.
$$
We immediately have
$$
\biggl|w-\sum_{j=1}^{N}b_j u_{\frac{1}{2}-j}\biggr|\leq C|x|^{-N}u_{-\frac{1}{2}} \text{ in }\mathbb{R}^2.
$$
by inverting the Kelvin transform and the proof is completed.
\end{proof}

As a corollary, for the solution from the previous proposition, we show that its Fourier coefficients along big circles vanish:
\begin{corollary}\label{fourier}
For the solution $u$ obtained by Proposition \ref{existence},  we have
$$\int_{\partial B_R}\biggl[u-\biggl(p+\sum_{j=1}^{N}b_ju_{\frac{1}{2}-j}\biggr)\biggr]\cdot\cos((i-\frac{1}{2})\theta)=0$$ 
for all $1\leq i\leq N\text{ and }R\geq M.$
\end{corollary}

\begin{proof}
	For simplicity, we denote the extended data by
	$$p_{ext}:=\biggl(p+\sum_{j=1}^{N}b_ju_{\frac{1}{2}-j}\biggr).$$
	From Proposition \ref{existence}, we have
	$
	\Delta(u-p_{ext})=0 \text{ in } \widehat{\{r>M\}}, \text{ and } u-p_{ext}=0 \text{ on } \widetilde{\{r>M\}}.
	$ For $R>M$ and each $i$ and $1\leq l\leq i$ fixed, define $v:=(r^{\frac{2l-1}{2}}-R^{2l-1}r^{-\frac{2l-1}{2}})\cos(\frac{2l-1}{2}\theta).$ Then 
	$v$ is a harmonic function in the slit domain $\mathbb{R}^2$. For any $L>R$, by integration by parts, we have
	$$
	0=\int_{B_L\setminus B_R}(u-p_{ext})\cdot\Delta v-\Delta(u-p_{ext})\cdot v=\int_{\partial(B_L\setminus B_R)}(u-p_{ext})_\nu\cdot v-(u-p_{ext})\cdot v_\nu.
	$$
	
	The asymptotic expansion of $u-p_{ext}$ yields that $|u-p_{ext}|=O(L^{-N-\frac{1}{2}})$, $|(u-p_{ext})_\nu|=O(L^{-N-\frac{3}{2}})$, $|v|=O(L^{\frac{2l-1}{2}})$ and $|v_\nu|=O(L^{\frac{2l-3}{2}})$ on $\partial B_L$ as $L\to\infty$. As a result,
	$$
	\int_{\partial B_L}(u-p_{ext})_\nu\cdot v-(u-p_{ext})\cdot v_\nu=O(L^{-N+l-1}).
	$$
	On $\partial B_R$, we have $v=0$ and $v_\nu=-(2l-1)R^{\frac{2l-3}{2}}\cos(\frac{2l-1}{2}\theta)$. Thus
	$$
	\int_{\partial B_R}(u-p_{ext})\cdot\cos(\frac{2l-1}{2}\theta)=O(R^{\frac{3-2l}{2}}L^{-N+l-1}).
	$$
	
	Sending $L\to\infty$ and then we finish our proof. 
\end{proof} 

Combining all of above, we complete the proof of Theorem \ref{start}

\section{Characterization of Symmetric Solutions}

In this section, we give the proof of Theorem \ref{SymSolIn2D} to characterize symmetric solutions and also show its perturbation version. We will first construct two important polynomials. Then we will use these auxiliary polynomials to show that $u$ and $v$ are actually half-space solutions to 2D thin obstacle problems with data $p$ and $q$ at infinity respectively. In the end, by applying Theorem \ref{TaylorExpansion}, we can fully characterize $u$ and $v$. We begin with the construction of two auxiliary polynomials as in Savin-Yu \cite{savin2021halfspace}.

\begin{lemma}\label{polynomials}
With the same assumption as in Theorem \ref{SymSolIn2D}, 
$$
P(t):=\operatorname{Re}(\partial_{x_1}u-i\partial_{x_2}u)^2(t,0)\text{ , } Q(t):=\operatorname{Re}(\partial_{x_1}v-i\partial_{x_2}v)^2(t,0)
$$
are two polynomials with degree $4k-3$. Moreover, we have
\begin{equation}\label{SymmetricPolynomial}
P(t)=-Q(-t).
\end{equation}
\end{lemma}
\begin{proof}
Recall that for simplicity, we define $b_j=b_j[a_1,a_2,...,a_{2k+1}]$ for $j=1,...,2k-2$, and
$$p_{ext}:=p+\sum_{i=1}^{2k-2}b_ju_{\frac{1}{2}-j}, \text{ and }
q_{ext}:=q+\sum_{i=1}^{2k-2}(-1)^jb_ju_{\frac{1}{2}-j}.
$$ 
Proposition \ref{existence} shows that 
$$|u-p_{ext}|+|v-q_{ext}|\leq 2M|x|^{-2k-\frac{1}{2}} \text{ in $\mathbb{R}^2$.}$$
Then it follows that 
\begin{equation}\label{GradientDecay}
	|\nabla u-\nabla p_{ext}|+|\nabla v-\nabla q_{ext}|\le C|x|^{-2k-\frac{3}{2}} \text{ for }|x|\geq 1.
\end{equation}

Note that $u$ is an entire solution to the thin obstacle problem of order $O(|x|^{2k-\frac{1}{2}})$ at infinity. Also note that the real part $(\partial_{x_1}u)^2-(\partial_{x_2}u)^2$ and the imaginary part $-2\partial_{x_1}u\partial_{x_2}u$ of $(\partial_{x_1}u-i\partial_{x_2}u)^2$ are both continuous in $\mathbb{R}^2$ and harmonic outside $\{x_2=0\}$. Thus they are harmonic functions in $\mathbb{R}^2$ and so is $(\partial_{x_1}u-i\partial_{x_2}u)^2$. Then by Liouville theorem, $(\partial_{x_1}u-i\partial_{x_2}u)^2$ is a polynomial of degree $4k-3$. By direct computation, we have
$$
(\partial_{x_1}p_{ext}-i\partial_{x_2}p_{ext})^2=\mathcal{P}(x_1+ix_2)+\sum_{j=1}^{4k-3}\mathcal{R}_j(x_1+ix_2),
$$
where $\mathcal{P}$ is a complex polynomial of degree $4k-3$, and each $\mathcal{R}_j$ is a $(-j)$-homogeneous rational function for $j=1,2,\dots,4k-3$.

With \eqref{GradientDecay}, it follows that  
$$
(\partial_{x_1}u-i\partial_{x_2}u)^2=\mathcal{P} \text{ in }\mathbb{R}^2.
$$ 
If we define 
$$
P(t):=\operatorname{Re}(\partial_{x_1}u-i\partial_{x_2}u)^2(t,0)=[(\partial_{x_1}u)^2-(\partial_{x_2}u)^2](t,0)=\operatorname{Re}\mathcal{P}(t),
$$
then $P$ is a real polynomial. Similarly, corresponding to $v$ and $q_{ext}$, we have 
$$
(\partial_{x_1}q_{ext}-i\partial_{x_2}q_{ext})^2=\mathcal{Q}(x_1+ix_2)+\sum_{j=1}^{4k-3}\mathcal{S}_j(x_1+ix_2),
$$
where $\mathcal{Q}$ is also a complex polynomial of degree $4k-3$, and $\mathcal{S}_j$ is a $(-j)$-homogeneous rational function for $j=1,2,\dots,4k-3$. Moreover, we have
$$Q(t):=\operatorname{Re}(\partial_{x_1}v-i\partial_{x_2}v)^2(t,0)=[(\partial_{x_1}v)^2-(\partial_{x_2}v)^2](t,0)=\operatorname{Re}\mathcal{Q}(t),
$$ 
which is also a real polynomial. With the symmetry condition $b_i[a_1,a_2,...,a_{2k+1}]=(-1)^i b_i[-a_1,a_2,...,-a_{2k+1}]$, it just follows from direct computation that 
\begin{equation}
	P(t)=-Q(-t).
\end{equation}
\end{proof}

With two auxiliary polynomials, we can show that $u$ and $v$ are in fact half-space solutions to \ref{2D}. Its proof is the key difference between our method and Savin-Yu's method in \cite{savin2021halfspace}.
\begin{proposition}\label{half-space}
With the same assumption as Theorem \ref{SymSolIn2D}, we have $u$ and $v$ are half-space solutions to \eqref{2D}.
\end{proposition}
\begin{proof}
With \eqref{SymmetricPolynomial}, we  show that up to a translation,  $u$ must be a half-space solution. Since $u=0$ on $\widetilde{\{r>M\}}$ according to Proposition \ref{start}, it suffices to show that $\text{spt}(\Delta u)$ has only one component.

Suppose, on the contrary, that 
$$
(-\infty,a]\cup [b,+\infty)\supset \text{spt}(\Delta u)\supset(-\infty,a]\cup [b,c] \text{ with }b>a,
$$
Here spt($\Delta u$) is the support of $\Delta u$ as a measure. For simplicity, we omit the second coordinate component since it is always $0$ in the proof. Note that the second component has to terminate in finite length since $\Delta u=0$ in $\widehat{\{r>M\}}$.

In $(-\infty,a]\cup [b,c]$, we have $\partial_{x_1}u=0$. Thus $P(t)=-(\partial_{x_2}u)^2\le 0$ for $t\in(-\infty,a]\cup[b,c]$. On the other hand, in $(a,b)$, $\partial_{x_2}u=0$ and $P(t)=(\partial_{x_1}u)^2\geq 0$. Moreover, since $u(a)=u(b)=0$, $u\geq0$ in $(a,b)$. we must have $\partial_{x_1}u(d)=0$ at some point $d\in(a,b)$ and there exists $\epsilon>0$ such that $\partial_{x_1}u>0\text{ on }(d-\epsilon,d)$ and $\partial_{x_1}u<0\text{ on }(d,d+\epsilon)$. Note that $u$ only has at most finitely many zeros on $(a,b)$. Otherwise, there are infinitely many zeros of $\partial_{x_1}u$ and thus of $P(t)$, which is impossible. So such $d$ and $\epsilon$ exist. Thus $d$ is a zero of $\partial_{x_1}u$ of odd multiplicity $\lambda$. So $d$ is a root of $P$ of multiplicity $2\lambda$.

With the symmetry described in \eqref{SymmetricPolynomial}, we have $Q<0$ on $(-b,-a)$ except finitely many zeros. This implies $v=0$ on $(-b,-a)$ except finitely many	points. By the continuity of $v$, we conclude that $v$ vanishes identically on $(-b,-a)$. So $\partial_{x_1}v=0\text{ on }(-b,-a)$ and $Q=-(\partial_{x_2}v)^2$.

Using the symmetry \eqref{SymmetricPolynomial} again, we have $-d$ is a zero of $Q$ and $\partial_{x_2}v$. Since $\partial_{x_2}v$ stays non-positive on $(-b,-a)$, $-d$ is a zero of $\partial_{x_2}v$ of even multiplicity $\mu$ and thus a zero of $Q$ of multiplicity $2\mu$, which is a contradiction with the symmetry.

As a result, $\text{spt}(\Delta u)$ must be a half line. A similar result holds for $\text{spt}(\Delta v).$  With \eqref{SymmetricPolynomial}, we see that if $\text{spt}(\Delta u)=(-\infty,a]$, then $\text{spt}(\Delta v)=(-\infty, -a].$
\end{proof}
\begin{remark}
We should point out that if we count the number of zeros of $P(t)$ and $Q(t)$, we can only prove the case of $p$ with leading term $u_{\frac{7}{2}}$ as in Savin-Yu \cite{savin2021halfspace} and $u_{\frac{11}{2}}$ (with slight modifications). However, this method does not work for general case since we cannot find all zeros of $P(t)$ and $Q(t)$. Instead, we check the multiplicity of the zero $d$ of $P(t)$ to get a contradiction, which is the main advantage of our paper. It is independent of the leading term.
\end{remark}

Finally, we can use the Lemma \ref{polynomials} and Proposition \ref{half-space} above to prove Theorem \ref{SymSolIn2D}.
\begin{proof}[Proof of Theorem \ref{SymSolIn2D}]

With the help of Proposition \ref{half-space}, we can apply Theorem \ref{TaylorExpansion} to get
	$$U_{-a}u=a_0'u_{2k-\frac{1}{2}}+\sum_{l=1}^{2k-2}\alpha_lu_{2k-\frac{1}{2}-i}+a_{2k+1}'u_{\frac{1}{2}}.$$
	
	Since $\sup\limits_{\mathbb{R}^2}|u-p|<+\infty$, we have $a_5'=0$ . The bound of $\biggl|u-\biggl(u_{2k-\frac{1}{2}}+\sum\limits_{l=1}^{2k-2}a_lu_{2k-\frac{1}{2}-l}\biggr)\biggr|$ in $\mathbb{R}^2$ yields $a_0'=1.$ Therefore, 
	$$U_{-a}u=u_{2k-\frac{1}{2}}+\sum_{l=1}^{2k-2}\alpha_lu_{2k-\frac{1}{2}-l}.$$
	Similarly, we have $$U_{a}v=u_{2k-\frac{1}{2}}+\sum_{l=1}^{2k-2}\beta_lu_{2k-\frac{1}{2}-l}.$$ 
	
	Thus $u$ and $v$ are symmetric. Finally, we can apply \eqref{SymmetricPolynomial} and obtain that $\alpha_l=(-1)^{l+1}\beta_l$. This finishes our proof of Theorem \ref{SymSolIn2D}.
\end{proof} 

With Theorem \ref{SymSolIn2D}, we obtain the following corollary, which allows us to perturb the symmetry condition 
$$
b_i[a_1,a_2,\dots,a_{2k-2},a_{2k-1}]=(-1)^{i+1}b_i[-a_1,a_2,\dots,a_{2k-2},-a_{2k-1}],
$$
for all  $1\leq i\leq 2k-2$, to almost-symmetric case. It is what we truly need when studying half-space $(2k-\frac{1}{2})$-homogeneous solutions. Its proof is almost the same as the Corollary B.2 of \cite{savin2021halfspace}.

\begin{corollary}\label{AlmostSymmetric2D}
	Given $p=u_{2k-\frac{1}{2}}+\sum\limits_{l=1}^{2k-1}a_lu_{2k-\frac{1}{2}-l}=u_{2k-\frac{1}{2}}+\sum\limits_{l=1}^{2k-1}\tilde{a}_lw_{2k-\frac{1}{2}-l}$ with $|a_j|\le 1$, we set 
	$$
	b_j^+:=b_j[a_1,a_2,\dots,a_{2k-2},a_{2k-1}]\text{, } b_j^-:=b_j[-a_1,a_2,\dots,a_{2k-2},-a_{2k-1}] $$ 
	for $j=1,2,...,2k-1,$ and
	$$
	p_{ext}:=p+\sum_{j=1}^{2k-2}b_ju_{\frac{1}{2}-j}=p+\sum_{j=1}^{2k-2}\tilde{b}_jw_{\frac{1}{2}-j}.
	$$ 
	
	Then there is a universal modulus of continuity $\omega$ such that 
	\begin{align*}
		&\sum_{l=1}^{2k-2}\biggl|\tilde{a}_l-\sum_{j=0}^{l}\frac{\alpha_{l-j}}{j!}\tau^j\biggr|+\biggl|\tilde{a}_{2k-1}-\sum_{j=1}^{2k-1}\frac{\alpha_{2k-1-j}}{j!}\tau^j\biggr|\\
		&+\sum_{j=1}^{2k-2}\biggl|\tilde{b}_{j}^{+}-\sum_{l=0}^{2k-2}\frac{\alpha_{2k-2-l}}{(l+j+1)!}\tau^{l+j+1}\biggr|\leq\omega(\sum_{j=1}^{2k-2}|b_j^+-b_j^-|)
	\end{align*} 
	for universally  bounded $\alpha_j$ (we naturally define $\alpha_0=1$) and $\tau$ such that $u_{2k-\frac{1}{2}}+\sum\limits_{l=1}^{2k-2}\alpha_lu_{2k-\frac{1}{2}-l}$ 
	is a solution to the thin obstacle problem in $\mathbb{R}^2$.
\end{corollary}

\begin{proof}
	Suppose there is no such $\omega$, then we can find a sequence $(a_{l}^{n})$ such that for any  bounded $\alpha_j$ and $\tau$ such that $u_{2k-\frac{1}{2}}+\sum\limits_{l=1}^{2k-2}\alpha_lu_{2k-\frac{1}{2}-l}$ 
	is a solution to the thin obstacle problem in $\mathbb{R}^2$, we have
	\begin{align}\label{SecondB2}
		&\sum_{l=1}^{2k-2}\biggl|\tilde{a}^{n} _{l}-\sum_{j=0}^{l}\frac{\alpha_{l-j}}{j!}\tau^j\biggr|+\biggl|\tilde{a}^{n}_{2k-1}-\sum_{j=1}^{2k-1}\frac{\alpha_{2k-1-j}}{j!}\tau^j\biggr|\nonumber\\
		&+\sum_{j=1}^{2k-2}\biggl|\tilde{b}^{+,n}_{j}-\sum_{l=0}^{2k-2}\frac{\alpha_{2k-2-l}}{(l+j+1)!}\tau^{l+j+1}\biggr|\geq\epsilon>0
	\end{align}
	for some $\epsilon>0$. Meanwhile, the corresponding $(b_j^{\pm,n})$ satisfy
	\begin{equation}\label{ThirdB2}
		\sum_{j=1}^{2k-2}|b_j^{+,n}-b_j^{-,n}|\to 0,
	\end{equation}
	as $n\to \infty$. Then up to subsequence, we have 
	$$a_l^n\to a_l^\infty \text{ and }b_j^{\pm,n}\to b_j^{\pm,\infty}.$$
	
	Set
	$$p_{n}^{+}=u_{2k-\frac{1}{2}}+\sum\limits_{l=1}^{2k-1}a^{n}_{l}u_{2k-\frac{1}{2}-l}=u_{2k-\frac{1}{2}}+\sum\limits_{l=1}^{2k-1}\tilde{a}^{n}_{l}w_{2k-\frac{1}{2}-l}$$
	and let $u^+_n$ be the (unique) solution to \eqref{2D} with data $p^+_n$ at infinity. Then by Proposition \ref{existence}, we have
	$$|u^+_n-p^+_n|\le M \text{ in }\mathbb{R}^2.$$ 
	
	Up to subsequence, we have $u^+_n$ converge to $u^+_\infty$ locally uniformly for some solution $u^+_\infty$ to the thin obstacle problem in $\mathbb{R}^2$. Moreover, we have 
	$$\biggl|u^+_\infty-\biggl[u_{2k-\frac{1}{2}}+\sum\limits_{l=1}^{2k-1}a^{\infty}_{l}u_{2k-\frac{1}{2}-l}\biggr]\biggr|\le M \text{ in $\mathbb{R}^2$.}$$ 
	Thus $u^+_\infty$ is the solution to \eqref{2D} with data $p^+_\infty=u_{2k-\frac{1}{2}}+\sum\limits_{l=1}^{2k-1}a^{n}_{l}u_{2k-\frac{1}{2}-l}$ at infinity. 
	
	Corollary \ref{fourier} implies that $b_j^{+,\infty}:=b_j[a_l^\infty]=\lim\limits_{n\to\infty} b_j^{+,n}.$ Simliarly, for $p_n^-=u_{2k-\frac{1}{2}}+\sum\limits_{l=1}^{2k-1}(-1)^{l}a^{n}_{l}u_{2k-\frac{1}{2}-l}$, we also have  $b_j^{-,\infty}:=b_j[(-1)^la_{l}^{\infty}]=\lim\limits_{n\to\infty} b_j^{-,n}.$ Invoking \eqref{ThirdB2}, we conclude 
	$$
	b_j[a_1,a_2,\dots,a_{2k-2},a_{2k-1}]=(-1)^{j+1}b_j[-a_1,a_2,\dots,a_{2k-2},-a_{2k-1}]
	$$
	Then Proposition \ref{half-space} yields that
	$$u_\infty^+=U_\tau(u_{2k-\frac{1}{2}}+\sum\limits_{l=1}^{2k-1}\alpha_{l}u_{2k-\frac{1}{2}-l})$$
	for $\alpha_j$ such that $u_{2k-\frac{1}{2}}+\sum_{l=1}^{2k-2}\alpha_lu_{2k-\frac{1}{2}-l}$
	is a solution to the thin obstacle problem in $\mathbb{R}^2$.	As a consequence, we have 
	\begin{align*}
		&\sum_{l=1}^{2k-2}\biggl|\tilde{a}^{\infty}_{l}-\sum_{j=0}^{l}\frac{\alpha_{l-j}}{j!}\tau^i\biggr|+\biggl|\tilde{a}^{\infty}_{2k-1}-\sum_{j=1}^{2k-1}\frac{\alpha_{2k-1-j}}{j!}\tau^j\biggr|\\
		&+\sum_{j=1}^{2k-2}\biggl|\tilde{b}_{j}^{+,\infty}-\sum_{l=0}^{2k-2}\frac{\alpha_{2k-2-l}}{(l+j+1)!}\tau^{l+j+1}\biggr|=0.
	\end{align*}
	Taking the limit of \eqref{SecondB2}, we get a contradiction.
\end{proof} 

\section*{Acknowledgment}
We would like to thank Prof. Hui Yu for the helpful discussion.
\printbibliography[title=References]

\end{document}